\documentclass[10pt]{amsart}

\usepackage{amsmath}
\usepackage{amsthm}
\usepackage{amsfonts}
\usepackage{amssymb,latexsym}
\usepackage[all]{xy}
\usepackage{graphicx}
\usepackage[mathscr]{eucal}

\theoremstyle{plain}
    \newtheorem{thm}{Theorem}[section]
    \newtheorem{prop}[thm]{Proposition}
    \newtheorem{lemma}[thm]{Lemma}
    
    \newtheorem{cor}[thm]{Corollary}

\theoremstyle{definition}

\theoremstyle{remark}

\newtheorem*{namedtheorem}{\theoremname}
\newcommand{\theoremname}{testing}

\newcommand{\stk}[1]{\stackrel{#1}{\longrightarrow}}

\newcommand{\ftwo}{\mathbb{F}_2}

\newcommand{\rar}{\ensuremath{\rightarrow}}

\newcommand{\n}{\noindent}

\DeclareMathOperator{\Hom}{Hom}

\DeclareMathOperator{\Ker}{Ker}

\newcommand{\Image}{\text{Image}}
\usepackage{verbatim}

\begin{document}

\title{Representations of The miraculous Klein group}
\date{\today}


\author{Sunil Chebolu}
\address{Department of Mathematics \\
Illinois State University \\
Normal, IL 61761  USA}
 \email{schebol@ilstu.edu}

\author{J\'{a}n Min\'{a}c}
\address{Department of Mathematics \\
University of Western Ontario \\
London, ON, N6A 5B7, Canada}
 \email{minac@uwo.ca}

\subjclass[2000]{}

\begin{abstract} The Klein group contains only four elements. Nevertheless this little group contains a number of remarkable entry points to 
current highways of modern representation theory of groups. In this paper, we shall describe all possible ways in which the Klein group can 
act on vector spaces over a field of two elements.  These are called representations of the Klein group. This description involves some powerful visual methods of representation theory which builds on the work  of generations of mathematicians starting roughly with the work of K. Weiestrass. We also discuss some applications to properties of duality and Heller shifts of the representations of the Klein group.
\end{abstract}

\maketitle



\section{Introduction}
Consider the familiar complex plane $\mathbb{C} = \{ x+ iy \, | \, x, y \text{ are real numbers} \}$ with two reflections
$\sigma$ and $\tau$ in the standard axes $X$ and $Y$ respectively.  Precisely, we have
\begin{eqnarray*}
\sigma(x+iy) & =  & x - iy, \text{ and } \\
 \tau(x+iy) &  = &  -x+iy.
 \end{eqnarray*}
 Thus, $\sigma$ is the complex conjugation and $\tau$ is like a real brother of $\sigma$.
 Note that if we apply $\sigma$ or $\tau$ twice, we get the identity map: $\sigma^{2} = 1 = \tau^{2}$.
 Also, we see that  $\sigma \tau = \tau \sigma = - 1$. Geometrically, the maps $\sigma \tau$
and $\tau \sigma$ are rotations by $180$ degrees in the complex plane. The set of maps
\[\{1, \sigma, \tau, \sigma \tau\}\]
forms a group under composition and is called the Klein four group or just Klein group, often denoted by $V_{4}$.
One would guess that the letter $V$ here is a sign of victory but  the reason is that ``\emph{Vier}'' in German means ``four.''
 ``Klein'' in German also means ``small'' and indeed Klein group $V_{4}$ having only four elements is quite small.
It is an absolutely amazing fact that this small and ostensibly innocent group contains  remarkable richness and that important mathematics can be developed by just studying this one group. The world's smallest field is $\ftwo = \{ 0, 1\}$, and one can think
of this as a toy model of complex numbers.  The problem to be investigated in this paper is the following: What are all the finite
dimensional representations of $V_{4}$ over $\ftwo$? That is, can one describe all possible actions of the group $V_{4}$ on finite dimensional vector spaces  $W$ over $\ftwo$. Although, we work over an arbitrary field of characteristic two, not much is lost if the reader assumes through out that the ground field $k$ is $\ftwo$. The reason for restricting to fields of characteristic $2$ is due to the fact that when the characteristic of the ground field $k$ is either zero or odd, the finite dimensional representations $W$ of $V_{4}$ have a very simple nature. Namely, $W$ is a sum of one dimensional representations. One each of these one-dimensional subspaces,  the generators $\sigma$ and $\tau$ act as multiplication by $1$ or  $-1$. We therefore  stick with fields of characteristic  $2$. This bring us to the world of  modular representations.  (That is, the characteristic of the field is a positive divisor of the order of the group.)

Note that $V_4$  is a product of two cyclic groups of order two. In terms
of generators and relations, $V_4$ has the following presentation.
\[ V_4 =   \langle \sigma, \tau \; | \; \sigma^2 = \tau^2 = 1,  \sigma \tau = \tau \sigma  \rangle .\]
The group algebra $kV_4$ is then isomorphic to
\[k[a , b]/(a^2, b^2),\]
where $a$ corresponds to  $1+ \sigma$ and $b$ to  $1+\tau$. We define an ideal $U$ of the $kV_4$ as the ideal generated by $a$ and $b$. This is an extremely important ideal called the augmentation ideal of our group ring. Sometime it will be convenient to divide $kV_4$ by ideal generated by $ab$. This simply amounts adding further relation $ab = 0$. The reason for this is that often $ab$ acts on our vector spaces as $0$ and therefore why not simplify our ring even further and add the relation $ab= 0$? We still call the image of $U$ in this new ring $U$ as we do not want to make our notation too complicated.  In this paper we present a
rather accessible proof of the well-known classification of all the finite dimensional  representations of $V_{4}$, or equivalently, of the finitely  generated indecomposable $kV_{4}$-modules.
These are also known as the modular representations of $V_{4}$. Note that a $V_{4}$ representation where $ab$ acts as zero can be viewed as a finite
dimensional $k$-linear space equipped with a pair commuting linear maps
$a$ and $b$ both of which square to zero.

Having explained what a representation of $V_{4}$ is, the following two questions have to be answered.
\begin{enumerate}
\item Why do we care about the representations of $V_{4}$?
\item What is unique about our approach?
\end{enumerate}

In answer to the first question, first note that groups act naturally on various algebraic objects including vector spaces, rings, algebraic
varieties and topological spaces.  These actions tend to be quite complex in general. Therefore it is important to find simple pieces of
this action and find ways to glue these pieces together to reconstruct the original action. Often this is related to other invariants of the
group or our given representation like cohomology groups and support varieties.
Amazingly, this goal in modular representation theory turns out to be exceedingly difficult.
 It turns out that besides the cyclic groups whose representations are very easily understood, Klein four group is one of the very few  (other groups are the dihedral groups) interesting yet non-trivial examples for which representation theorists are able to completely classify all the finite dimensional modular representations. There is a lot to be learned
by studying the representation theory of this one group and it goes to tell how complex the
study of modular representations can be for an arbitrary group.

Now we turn to the second question. Although the classification of the finite dimensional representations of $V_{4}$ is well-known
and many proofs can be found in the literature, we could not find a proof to our heart's content. This is what motivated us to write up one -- one that is transparent and which takes a minimal background.  Furthermore, our approach is diagrammatic, so the reader
can see what is happening through pictures. These methods, besides making the statements of
theorems and proofs elegant and conceptual, give a better
insight into the subject. We mostly follow Benson's approach  \cite{ben-1}  but we approach some parts of his proof from a different point of view and simplify  them and in particular we make our
proof accessible for a general reader. One ingredient that is new in our approach is Auslander-Reiten sequences which will be introduced later in the paper.

The subject of classifying the indecomposable representations of the Klein group has a long and rich history that can be traced all the way back to V. A. Ba\v{s}ev \cite{basev}, a student of I.R. \v{S}afarevi\v{c}, who observed that an old result of L. Kronecker on pairs of matrices can be used effectively in the classification, but over algebraically closed fields.  This result of L. Kronecker on pairs of matrices was actually a completion of the work of K. Weierstrass. Then later on
I. M. Gelfand and V. Ponomarev \cite{Gel-Pon} observed in their analysis of the representations of the
Lorentz group that quiver techniques were quite useful and they both knew that   G.Szekeres had a result in this direction. However,  they did
not know  enough details about Szekeres's techniques and therefore they invented their new innovative and influencial quiver method  which is influenced by Maclane's notion of relations -- a generalization of a linear map.   In  \cite{hell-rei} A. Heller and I. Reiner provided another nice approach to the classification where they also worked over fields that are not necessarily algebraically closed.  Finally D. Benson \cite{ben-1}
wrote a modern treatment of the classification of the indecomposable representations of the Klein group in which he combined some of the crucial ideas in the works of the aforementioned people. The diagrammatic methods in our paper are inspired by S. B. Conlon  who  introduced these in \cite{conlon}.
It is quite remarkable that a complete understanding of an innocent looking group on four elements would take the works of some of the great minds of the 19th, 20th, and 21st centuries.

 Before going further, we remind the reader some basis facts and terminology. We refer the reader to Carlson lecture notes for basic representation theory \cite{carlson-modulesandgroupalgebras}. In the category of modules over a the Klein group (or more generally, over a $p$-group), the three terms
``injective", ``projective" and ``free" are synonymous.
Given a $V_{4}$-module $M$, its Heller shift $\Omega( M)$ is defined to
be the kernel of a minimal projective cover of $M$.  It can be shown that minimal projective covers are
unique up to isomorphism and from that it follows that $\Omega(M)$ is well-defined. Inductively one defines $\Omega^n (M)$ to be
$\Omega( \Omega^{n-1} M)$. Similarly, $\Omega^{-1} M$ is defined to
be the cokernel of an injective envelope of $M$, and $\Omega^{-n}
(M)$ to be $\Omega^{-1}( \Omega^{-n+1} M)$. Again one can shown that these
are well-defined modules. The modules
$\Omega^i M$ are also known as the syzygies of $M$.
By the classical Krull-Remak-Schmidt theorem, one knows that every
representation of a finite group decomposes as a direct sum of
indecomposable ones. Thus it suffices to classify the indecomposable
representations.

Advice for the novice: some arguments in our paper are only sketched and some notions maybe still unfamiliar for a novice. If that is the case, we advice readers  to skip these parts on the first reading as they may became more clear later on and they most likely will not influence the basic understanding of the key ideas.
The main point of this article is to provide overview of the remarkable proof of classification of representations of Klein group $V_4$ with  appreciation of the works of number of people and to show that this proof open doors to study modern
group representation where Auslander-Reiten sequences play increasingly important role. We hope that after reading our article a reader will read more texts in the references and possibly go on to further exciting heights in group representation theory.

\section{Indecomposable Representations of Klein's four group}

We list  all the indecomposable representations of $V_4$ below. Note that these are just the finitely generated modules over the group algebra
\[kV_{4} \cong k[a, b]/(a^{2}, b^{2})\]
which cannot be written as a sum of strictly smaller modules (much the same way prime numbers cannot be written as product of
smaller numbers).
 Since we take a diagrammatic approach, we first  explain the diagrams that follow.
Each bullet represents a one dimensional $k$ vector space, a
southwest arrow "$\swarrow$" connecting two bullets corresponds to
the action of $a$ and maps one bullet to the other in the indicated
direction, and similarly the south east arrows "$\searrow$"
correspond to the action of $b$. If no arrow emanates from a bullet
in given direction, then the corresponding linear action is
understood to be zero.

\begin{thm}\emph{(Kronecker, Weierstrass, Basev, Gelfand,  Ponomarev, Conlon, Heller, Reiner, Benson)} \, \cite{Gel-Pon, conlon, hell-rei, basev, ben-1} Let $k$ be a field of characteristic $2$. Every isomorphism class of an indecomposable $V_4$ representation over $k$ is
precisely one of the following.
\begin{enumerate}
  \item   The projective indecomposable module $kV_4$ of
  dimension 4.
   
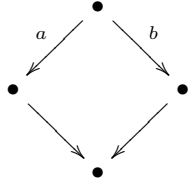
\[
    \xymatrix@=2em{
     & \bullet \ar[dl]_a \ar[dr]^b & \\
     \bullet \ar[dr]&  & \bullet \ar[dl]\\
     & \bullet &
    }
   \]
  \item   The (non-projective) indecomposable even dimensional modules:
   \begin{enumerate}
     \item  For each even dimension $2n$ and an indecomposable rational canonical from corresponding
      to the power of an irreducible monic polynomial $f(x)^l = \sum_{i=0}^n \theta_i x^i$, $(\theta_n = 1)$ there is an indecomposable
     representation given by
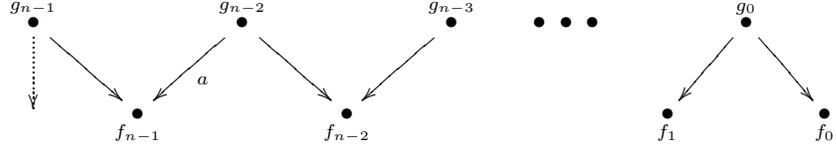
\[
\xymatrix@=1.5em{
    &  & \overset{g_{n-1}}{\bullet} \ar@{..>}[d] \ar[dr]& & \overset{g_{n-2}}{\bullet}\ar[dl]^a \ar[dr] &  & \overset{g_{n-3}}{\bullet} \ar[dl]  & \bullet  \hspace{1 mm}
      \bullet \hspace{1 mm} \bullet
      & & \overset{g_0}{\bullet} \ar[dl] \ar[dr] \\
    &  & & \underset{f_{n-1}}{\bullet} & & \underset{f_{n-2}}{\bullet} &  & &  \underset{f_1}{\bullet} & &
      \underset{f_0}{\bullet}
    }
\]
where $a(g_{n-1}) = \sum_{i = 0}^{n-1} \theta_i f_i$, as represented by the vertical dotted arrow emanating from $g_{n-1}$ above.
\item For each even dimension $2n$ there is an indecomposable
representation given by
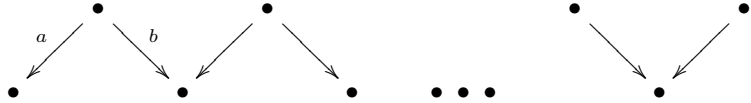
\[
     \xymatrix@=2em{
& & \bullet \ar[dr]^b \ar[dl]_a & & \bullet \ar[dl] \ar[dr] & & &
     \bullet \ar[dr] & & \bullet \ar[dl]   \\
 &     \bullet &  & \bullet & & \bullet & \bullet  \hspace{1 mm}
      \bullet \hspace{1 mm} \bullet & &  \bullet &
}
\]
\end{enumerate}
\item The (non-projective) indecomposable odd dimensional modules:
\begin{enumerate}
\item The trivial representation $k$.
\[ \bullet \]
\item For each odd dimension $2n+1$ greater than one, there is an indecomposable
    representation given by
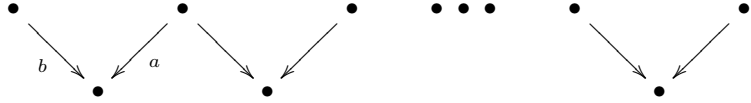
\[
\xymatrix@=2em{
    &  \bullet \ar[dr]_b & & \bullet\ar[dl]^a \ar[dr] &  & \bullet \ar[dl]  & \bullet  \hspace{1 mm}
      \bullet \hspace{1 mm} \bullet & \bullet \ar[dr]
      & & \bullet \ar[dl] \\
    &   & \bullet & & \bullet &  & &  & \bullet &
    }
\]
\item For each odd dimension $2n+1$ greater than one, there is an indecomposable
representation given by
     
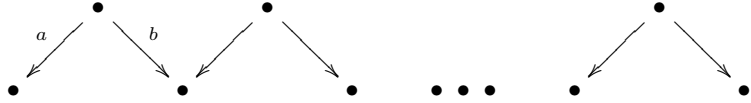
\[
     \xymatrix@=2em{
& & \bullet \ar[dr]^b \ar[dl]_a & & \bullet \ar[dl] \ar[dr] & &
      & & \bullet \ar[dl] \ar[dr] &  \\
 &    \bullet &  & \bullet & & \bullet & \bullet  \hspace{1 mm}
      \bullet \hspace{1 mm} \bullet  &  \bullet & & \bullet
}
\]
  \end{enumerate}
\end{enumerate}

\end{thm}

 The reader may decide to make a pleasant check that the above diagrams are indeed representations of $V_{4}$.   For each $V_{4}$-module $M$, one can 
 define the dual $V_{4}$-module $M^{*}$, where $M^{*}$ is the dual $k$-vector space of  $M$, and a group element $\sigma$ of $V_{4}$ acts on $f$ in $M$ via the rule $\sigma f(m) = f(\sigma^{-1} m)$.
 Then, as a fun exercise we ask the reader
 to verify that the diagrams in 3(a) and 3(b) are dual to
each other. This will help the reader to get acquainted with some of
the diagrammatic methods that will appear later on.  We now begin by
proving the easy part of the theorem.

\begin{lemma} \label{lemma0} All representations of $V_4$ that appear in the above
theorem are indecomposable and pair-wise non-isomorphic.
\end{lemma}

\begin{proof} Item (1) is not isomorphic to the rest because it is
the only module that contains a non-zero element $x$ such that
$(ab)x \ne 0$. Modules in item (2) are even dimensional and those in
(3) are odd dimensional and hence there is no overlap between the
two. To see that all the $2n$ dimensional representations of item
$2(a)$ are distinct, it is enough to observe that the rational
canonical forms of the linear transformations on the co-invariant
submodules,
 \[ b^{-1}a : M/UM \rar M/UM \]
are distinct, where $U$ is the ideal generated by $a$ and $b$. To see that the $2n$ dimensional representation of
item $2(b)$ does not occur in item $2(a)$ observe that kernel of the
$b$-action in both cases have different dimensions: $n$ for the
module in $2(a)$ and $n+1$ for that in 2$(b)$. The two $2n+1$
dimensional modules in items $3(b)$ and $3(c)$ are non-isomorphic
because it is clear from the diagrams that the dimensions of the
invariant submodules in both cases are different: $n$ for those in
item $3(b)$, and  $n+1$ for those in item $3(c)$.
\end{proof}

Of course the hard thing is to show that every indecomposable
representation of $V_4$ is isomorphic to one  in the above list.
Since projective modules over $p$-groups are free,
 there is only one indecomposable projective $V_4$-module,
namely $kV_4$ which occurs as item (1) in the list. Therefore we
only consider indecomposable projective-free (modules which do not
have projective summands) $V_4$-modules.

One can get a better handle on the projective-free representations
of $V_4$ by studying the representations of the  so called Kronecker
Quiver, which is a directed graph $Q$ on two vertices as shown
below.
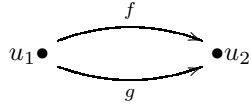
\[
\xymatrix{
 u_1 \bullet \ar@/^10pt/[rr]^f   \ar@/_10pt/[rr]_g &  & \bullet u_2
}
\]
A representations of the above quiver is just a pair of finite
dimensional $k$-vector spaces $V$ and $W$ and a pair of $k$-linear
maps $\psi_1$ and $\psi_2$ from $V$ to $W$. Such a representation
will be denoted by the  four tuple $[V, W ; \psi_1, \psi_2]$, and
given two such representations, the notion of  direct sum, and
morphisms between them are  defined in the obvious way.  Thus it
makes sense to talk about the isomorphism class of an indecomposable
representation of $Q$. Let us call a representation of $Q$ special
if the following conditions hold:
\begin{eqnarray*}
\Ker(\psi_1) \;  \bigcap \; \Ker(\psi_2)  & = & 0 \\
\Image(\psi_1) + \Image(\psi_2) & = & V_2.
\end{eqnarray*}

\begin{prop} \label{prop1} \cite{hell-rei} There is a one-one correspondence
between  the isomorphism classes of (indecomposable) projective-free
representations of $V_4$ and those of the special  (indecomposable)
representations of the Kronecker quiver. Under this correspondence,
an (indecomposable) projective-free representation $M$ of $G$
corresponds to the (indecomposable) representation  of $Q$ that is
given by $[M/UM, UM; a, b]$. Conversely, given an (indecomposable)
special representation $[V, W ; \psi_1, \psi_2]$ of $Q$, the
corresponding (indecomposable) $G$-module $M$ is given by $M = V
\oplus W$ where $a(\alpha, \beta) := (0, \psi_1(\alpha))$ and
$b(\alpha, \beta) := (0, \psi_2(\alpha)).$
\end{prop}

We will use this translation between the representations of the
Klein group and the Kronecker Quiver freely through out the paper.

If $M = [V_1, V_2; a, b]$ is an indecomposable projective-free
representation, then we have
\begin{align*}
V_1 &= 0 \; \Leftrightarrow \; M = 0 \\
   V_2 &= 0\; \Leftrightarrow \; M = k.
\end{align*}
So henceforth it will be assumed that the spaces $V_1$ and $V_2$ are
non-zero, i.e., $M$ is an indecomposable projective-free and a
non-trivial representation of $V_4$.

We begin with some lemmas  that will help streamline the proof of
the classification theorem.  The proofs of these lemmas will be deferred to the last section.
It should be noted that these lemmas are also of independent interest.

\begin{lemma} \cite{ben-1}\label{lemma1} Let $M$ be a projective-free $V_4$-module given by $[V_1, V_2; a, b]$.
Then we have the following.
\begin{enumerate}
\item $M$ contains a copy of $\Omega^l(k)$ for some positive integer $l$ if and
only if the transformation
\[a + \lambda b : V_1 \otimes_k k[\lambda] \rar V_2 \otimes_k k[\lambda]\]
is singular, i.e, $det( a + \lambda b) = 0$.
\item Dually, $\Omega^{-l} (k)$ is a quotient of $M$ for some
positive integer $l$ if and only if the transformation
\[ a^* + \lambda b^* : V_2^* \otimes_k k[\lambda] \rar V_1^* \otimes_k k[\lambda]  \]
is singular.
\end{enumerate}
\end{lemma}

The next lemma is very crucial to the classification. To the best of our knowledge, nowhere in the literature is this
lemma stated explicitly, although it is secretly hidden in Benson's proof of the classification \cite{ben-1}.
We use Auslander-Reiten sequences to give a transparent proof of this lemma in the last section.

\begin{lemma} \label{lemma2} Let $M$ be a projective-free $V_4$-module. Then we have
the following.
\begin{enumerate}
  \item If $l$ is the smallest positive integer such that
  $\Omega^l(k)$ is isomorphic to a submodule of
  $M$, then $\Omega^l(k)$ is a summand of $M$.
  \item Dually, if $l$ is the smallest positive integer such that
  $\Omega^{-l}(k)$ is isomorphic to a quotient module of
  $M$, then $\Omega^{-l}(k)$ is a $G$-summand of $M$.
\end{enumerate}
\end{lemma}

\begin{lemma} \cite{johnson-1}\label{lemma3} For all integers $n$, $\Omega^n (k)$ is isomorphic to the dual
representation  $\Omega^{-n} (k) ^*$. Furthermore,
\begin{enumerate}
  \item If $n$ is positive, then $\Omega^n (k)$ is a $2n+1$ dimensional indecomposable representation
  given by
  
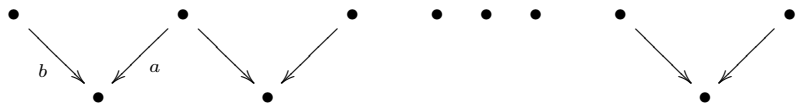
\[
\xymatrix@=2em{
    &  \bullet \ar[dr]_b & & \bullet\ar[dl]^a \ar[dr] &  & \bullet \ar[dl]  & \bullet  \hspace{4 mm}
      \bullet \hspace{4 mm} \bullet & \bullet \ar[dr]
      & & \bullet \ar[dl] \\
    &   & \bullet & & \bullet &  & &  & \bullet &
    }
\]
\item If $n$ is a negative integer, then $\Omega^n (k)$ is a $2n+1$
dimensional indecomposable representation given by
     
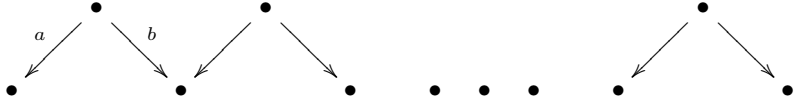
\[
     \xymatrix@=2em{
& & \bullet \ar[dr]^b \ar[dl]_a & & \bullet \ar[dl] \ar[dr] & &
      & & \bullet \ar[dl] \ar[dr] &  \\
 &    \bullet &  & \bullet & & \bullet & \bullet  \hspace{4 mm}
      \bullet \hspace{4 mm} \bullet  &  \bullet & & \bullet
}
\]
\end{enumerate}
\end{lemma}

We now give the proof of the classification theorem assuming these
lemmas. The lemmas will be proved in the last section.  Let $M = (V_1,
V_2; a, b)$ be an indecomposable projective-free  representation of
$V_4$. We will show that $M$ is isomorphic to one of the
representation that appear in items (2) it is even dimensional, and
to those in item (3) if it odd dimensional.

\subsection{Even dimensional representations} Let $M$ be an even
dimensional ($2n$ say) indecomposable representation. We break the
argument into cases for clarity. \\

\n \emph{Case 1:} $\det (a + \lambda b)$ is non-zero. We have two
subcases. First assume that $\det b \ne 0$. Then consider the map
\[ b^{-1}a : V_1 \longrightarrow V_1.\]
We claim that this map is indecomposable. Suppose we have a
decomposition $f \oplus g$ of $b^{-1}a$ as follows
\[
\xymatrix{ V_1 \ar[r]^{b^{-1}a} \ar[d]_{\cong} & V_1 \ar[d]^{\cong}\\
A \oplus B \ar[r]_{f \oplus g} & A \oplus B
 }
\]
Set $C:= b(A)$ and $D:= b(B)$. Then it is straightforward  to verify
that $M = (V_1, V_2; a, b)$ decomposes as
\[ (A, C; a|_{A}, b|_{A}) \; \bigoplus \; (B, D; a|_{B}, b|_{B}).\]
Since $M$ is indecomposable it follows that the map  $b^{-1}a$ is
indecomposable. Now since $b^{-1}a$ is indecomposable we can choose
a basis $\{g_0, g_1, \cdots, g_{n-1}\} $ of $V_1$ such that the
rational canonical form of $b^{-1}a$ has only one block which
corresponds to some power of an irreducible polynomial $f(x)^r =
\sum_{i = 0}^{n-1} \theta_i x^i$. This means we have
\begin{eqnarray*}
b^{-1}a \; (g_i) &= & g_{i+1}  \hspace{3 mm} \text{for} \;\; 0 \le i \le n-2, \\
b^{-1}a \; (g_{n-1}) &=& \sum_{i=0}^{n-1} \theta_i g_i. \hspace{3
mm} (*)
\end{eqnarray*}

Now the vectors $f_i := b(g_i)$ for $0 \le i \le n-1$ define a basis
for $V_2$ because $b$ is non-singular. With respect to the bases
$(g_i)$ of $V_1$ and $(f_i)$ of $V_2$, it  is now clear that $M$ has
the description
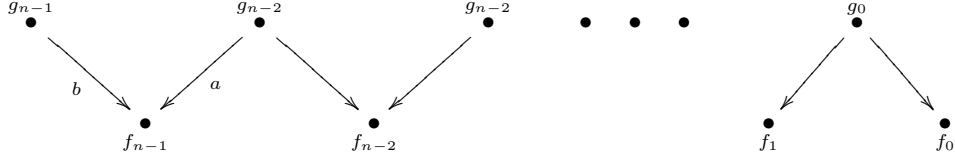
\[
\xymatrix@=1.9em{
       \overset{g_{n-1}}{\bullet} \ar[dr]_b & & \overset{g_{n-2}}{\bullet}\ar[dl]^a \ar[dr] &  &
      \overset{g_{n-2}}{\bullet} \ar[dl]  & \bullet  \hspace{4 mm}
      \bullet \hspace{4 mm} \bullet
      & & \overset{g_0}{\bullet} \ar[dl] \ar[dr] \\
       & \underset{f_{n-1}}{\bullet} & & \underset{f_{n-2}}{\bullet} &  & &  \underset{f_1}{\bullet} & &
      \underset{f_0}{\bullet}
    }
\]
The action of $a$ on $g_{n-1}$ can be seen by applying $b$ on both
sides of the equation (*) above: $a(g_{n-1}) = \sum_{i=0}^{n-2}
b(g_i) = \sum_{i=0}^{n-1} f_i$. These representations are the
exactly ones in item 3(a).

Now if $\det(b)  = 0$, we do a change of coordinate trick.
We assume that $k$ is an infinite field. If $k$ is finite, we can pass to an extension field and do a descent argument; see \cite{ben-1} for details.   Then there exists some
$\lambda_0$ in $k$ such that $\det(a + \lambda_0 b) \ne 0$.  Now
consider the tuple $(V_1, V_2; b, a+\lambda_0 b)$. By case (i), we
know that there exist bases for $V_1$ and $V_2$ such that $a +
\lambda_0 b = I$ and $b = J_0$ (the rational canonical form \cite{LinearAlgebra} \footnote{Most readers are familiar with the Jordan canonical form
of an operator acting on a vector space over $\mathbb{C}$ or other algebraically closed fields. These forms use critically the fact that non-constant polynomials have roots.  However, a parallel and beautiful theory also exists when the field is not algebraically closed, and this is not so well-known. One often thinks about the base field as the field of rational numbers and the name ``The rational canonical form'' stick also to completely different fields including $\ftwo$. }
corresponding to any indecomposable singular transformation). This
gives the representation in item 2(b).
\\

\n \emph{Case 2:} $\det (a + \lambda b) = 0$. We will show that this
case cannot arise.  First suppose that there is a copy of
$\Omega^{l} (k)$ in $M$ for some positive integer $l$. Now pick $l$
to be the smallest such integer, then by lemma \ref{lemma2} we know
that $\Omega^l (k)$ is a direct summand of $M$. Since $M$ is
indecomposable, this means $M$ has to be isomorphic to
$\Omega^l(k)$, which is impossible since the latter is odd
dimensional while $M$ was assumed to be even dimensional. So the
upshot is that $M$ does not contain  $\Omega^l(k)$ for any positive
$l$. By lemma \ref{lemma1} this is equivalent to the fact $\det (a +
\lambda b) \ne 0$ in the ring $k[\lambda]$ which contradicts our
hypothesis.

\subsection{Odd dimensional representations} If $M$ is odd dimensional,
then clearly $\dim V_1 \ne \dim V_2$.  We consider the two cases.\\

\n \emph{Case 1:} $\dim V_1 > \dim V_2$. Then there is a non-zero
vector $\omega( \lambda)$ in $ V_1 \otimes_k K[\lambda]$ such that
$(a + \lambda b) (\omega (\lambda)) = 0$ which then implies, by
lemma \ref{lemma1}, the existence of a copy of $\Omega^l(k)$ inside
$M$ for some $l > 0$. Picking $l$ to be minimal, we can conclude
from lemma \ref{lemma2} that $\Omega^l(k)$ is a direct summand of
$M$. Since $M$ is indecomposable, we have $M \cong \Omega^l(k)$.
This gives the modules in item 3(b).\\
\n \emph{Case 2 :} $\dim V_1 < \dim V_2$. Dualising $M = (V_1, V_2;
a , b)$, we get the dual representation   $M^* = (V_2^*, V_1^*; a^*,
b^*)$ which is also indecomposable. Now $\dim V_2^* > \dim V_1^*$,
so by Case(1) we know that $M^* \cong \Omega^l(k)$ for some $l$
positive. Taking duals on both sides and invoking lemma
$\ref{lemma3}$, we get $M \cong \Omega^{-l}(k)$.  This recovers the
modules in item 3(c).

This completes the proof of the classification of the indecomposable
representations of $V_4$.

\section{Some applications}

Having a good classification of the indecomposable representations of a finite group
helps a great deal in answering general module theoretic questions.
In this section, we illustrate this by proving some facts about module over the Klein group.  Note that we don't know of any
direct proofs of the statements below  that  do not depend on the classification of the indecomposable representations.

\subsection{Heller Shifts of the $V_4$-representations.}
We will show how our knowledge of the representations of $V_4$ can
be used to give a homological characterisation of the parity of the
dimensions of the representations. Proofs of the propositions are given in the last section.

\begin{prop} \label{prop:heller} \cite{ring} If $M$ is an even dimensional indecomposable
projective-free representation of $V_4$, then $\Omega(M) \cong M$.

\end{prop}

\begin{cor}  A finite dimensional  projective-free representation $M$ of
$V_4$ is even dimensional if and only if $\Omega(M) \cong M$.
\end{cor}

\begin{proof} We only have to show that if $M$ is an odd dimensional
indecomposable then $\Omega(M) \ncong M$. By the classification
theorem, we know that $M$ is isomorphic to $\Omega^l(k)$ for some
integer $l$. Then $\Omega(M) \cong \Omega(\Omega^l(k)) \cong
\Omega^{l+1}(k)$, which is clearly not isomorphic to $M$ just for
dimensional reasons: just note that dimension of $\Omega^{n} (k)$
 is $2n+1$.
 \end{proof}

\subsection{Dual representations of $V_4$} We will use our knowledge
of the representations of $V_4$ to characterise the parity of the
dimension of a representation using the concept of self-duality.

Recall that if $M$ is a finite dimensional representation of a group
$G$, then one can talk about the dual representation $M^* := \Hom(M,
k)$, where a group element $g$ acts on a linear functional $\phi$ by
$(g \cdot \phi)(x) := \phi(x g^{-1})$. A representation of $G$ is
self-dual if it is isomorphic to its dual.

When  $G = V_4$, it is not hard to see that if $M = (V_1, V_2; a,
b)$ is a projective-free representation of $V_4$, then $M^* =
(V_2^*, V_1^*; a^*, b^*)$.

\begin{prop} \label{prop:dual} Even  dimensional indecomposable representations of
$V_4$ are self-dual.
\end{prop}

\begin{cor} A non-trivial indecomposable representation of $V_4$ is
even dimensional if and only if it is self-dual.
\end{cor}

\begin{proof} If $M$ is a non-trivial odd dimensional representation
of $V_4$, then we know that $M \cong \Omega^l(k)$ for some $l \ne
0$. Then $M^* \cong (\Omega^l(k))^* \cong \Omega^{-l}(k)$. In
particular, $M^* \ncong M$.

\end{proof}

\section{Proofs}

In this section we give the proofs of the lemmas and propositions
that were used in the classification theorem and applications.

\subsection{Proof of proposition \ref{prop1}}
Let $M$ be  a projective-free $V_4$ module. Then we have
we have $ab(M) = 0$, it follows that $UM$ is included in $U^V_4$. Remarkably one can show that if $M$ is additionally not trivial representation and $M$ is indecomposable then $UM$ is actually equal $M^V_4$, the $V_4$ invariant
submodule of $M$. Consider  short exact sequence of $V_4$ modules
\[ 0 \rar UM \rar M \rar M/UM \rar 0. \]
Let $\pi: M \rar UM$ be a vector space retraction  of the inclusion
$UM \hookrightarrow M$. Define a $V_4$ action on the vector space
$M/UM \oplus M$ as follows:
\begin{align*}
a(x, y) := (0, ax) \\
b(x, y) := (0, by).
\end{align*}
Then it is easy to verify that the map $x \mapsto (x, \pi(x))$
establishes an isomorphism of $V_4$ modules between $M$ and $M/UM
\oplus UM$. Thus $M$ is determined by the vector spaces $M/UM$ and
$UM$ and the linear maps $a, b: M/UM \rar M$. This data amounts to
giving a special representation of $Q$.

In the other direction, suppose $[V_1, V_2; \psi_1, \psi_2]$ is a
special representation of $Q$. Define a $V_4$ action on the vector
space $V_1 \oplus V_2$ by setting $a(x, y):=(0, \psi_1(x))$ and
$b(x, y):= (0, \psi_2(x))$. This is easily shown to be a projective
free $V_4$ module.

We leave it as an exercise to the reader to verify that the recipes
are inverses to each other.

It is also clear that these recipes respect direct sum of
representations. Thus the indecomposables are also in 1-1
correspondence.

\subsection{Proof of lemma \ref{lemma1}}

Suppose $M$ contains a copy of $\Omega^{l}(k)$, for some $l \ge 1$.
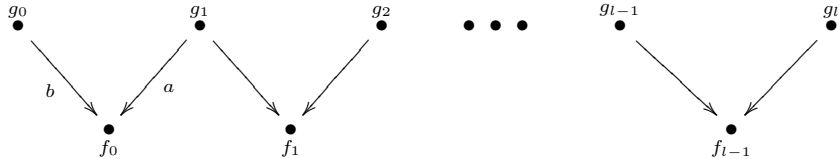
\[
\xymatrix@=2em{
    &  \overset{g_0}{\bullet} \ar[dr]_b & & \overset{g_1}{\bullet} \ar[dl]^a \ar[dr] &  & \overset{g_2}{\bullet} \ar[dl]  & \bullet  \hspace{1 mm}
      \bullet \hspace{1 mm} \bullet & \overset{g_{l-1}}{\bullet} \ar[dr]
      & & \overset{g_l}{\bullet} \ar[dl] \\
    &   & \underset{f_0}{\bullet} & & \underset{f_1}{\bullet} &  & &  & \underset{f_{l-1}}{\bullet} &
    }
\]
Define a vector $V(\lambda): = g_0 + g_1 \lambda + g_2 \lambda^2 +
\cdots + g_l  \lambda^l$.  A trivial verification shows that $(a +
\lambda b) (V(\lambda)) = 0$ and therefore $a + \lambda b$ is a
singular transformation as desired.

Conversely, suppose $a + \lambda b$ is singular. Then there is a
non-zero vector $V(\lambda) =  g_0 + g_1 \lambda + g_2 \lambda^2
\cdots g_l \lambda^l$ of smallest degree $l$ in $V_1 \otimes
k[\lambda]$ (so $g_l \ne 0$) such that $(a + b \lambda) (V(\lambda))
= 0$. This means: $a(g_0) = 0$, $b(g_i) = a(g_{i+1})$ for $0 \le i
\le l-1$, and $b(g_l) = 0$. We now argue that these equations give a
copy of $\Omega^l(k)$ inside $M$. To this end, it suffices to show
that the vectors $\{g_0, g_1, g_2, \cdots, g_l \}$ are linearly
independent. As a further reduction, we claim that it suffices to
show that $\{ a(g_1),  a(g_2), \cdots, a(g_l) \}$ are linearly
independent. For, then it will be clear that  $\{ g_1, g_2, \cdots,
g_l \}$is linearly independent, and moreover if $g_0 = \sum_{i=1}^l
c_i \, g_i$, applying $a$ on both sides we get $a(g_0) = 0 = \sum_{i
= 1}^l c_i \, a(g_i)$. Linear independence of $a(g_i)$ forces all
the $c_i = 0$. Thus we will have shown  that $\{g_0, g_1, g_2,
\cdots, g_l \}$ is linearly independent. So it remains to establish
our claim that $\{ a(g_1), a (g_2), \cdots, a (g_l) \}$  is a
linearly independent set. Suppose to the contrary that there is a
non-trivial linear combination of $a(g_i)$'s which is zero: say
$\sum_{i=1}^l \gamma_i \, a (g_i) = 0$ (*). We will get a
contradiction by showing that there is a vector of smaller degree ($
< l$) in $\Ker(a + \lambda b)$. It suffices to produce elements
$(\tilde{g}_i)_{0 \le i \le l-1}$ such that $a(\tilde{g}_0) = 0 $,
$b(\tilde{g_{l-1}}) = 0$, and for $0 \le i \le l-2$, $b(\tilde{g}_i)
= a (\tilde{g_{i+1}})$ ($\diamond$). For then the vector
$\sum_{i=0}^{l-1} \tilde{g}_i \lambda^i$ will be of degree less than
$l$ belonging to the kernel of $a + \lambda b$. To start, we set
$\tilde{g}_0 = \sum_{i=0}^l \gamma_i\, g_i$. The condition $a
(\tilde{g}_0) = 0$ is satisfied by assumption (*). Now define
$\tilde{f_0} := b( \tilde{g}_0) = \sum_{i=1}^l \gamma_i \, b(g_i) =
\sum_{i = 1}^{l-1} \gamma_i \, b(g_i)$ (since $b(g_l) = 0$). Then we
define $\tilde{g}_1 = \sum_{i=0}^{l-1} \gamma_i \, g_{i+1}$ so that
we have the required condition $a(\tilde{g}_1) = b(\tilde{g}_0)$.
Now we simply repeat this process: Inductively we define, for $ 0
\le t \le l-1 $,
\begin{eqnarray*}
\tilde{g}_t &= &\sum_{i=1}^{l-t} \gamma_i \;g_{i+t}, \\
\tilde{f}_t &= & \sum_{i=1}^{l-t} \gamma_i\; b(g_{i+t}).
\end{eqnarray*}
When $t = l-1$, we have $\tilde{g}_{l-1} = \gamma_1 g_l$ and
$\tilde{f}_{l-1} = 0$. So this inductive process terminates at $t=
l-1 (< l) $ and the requirements $(\diamond)$ are satisfied by
construction. Thus we have shown that the vector $\sum_{i=0}^{l-1}
\tilde{g}_{i}\, \lambda^i$ is of smaller degree in the kernel of $a
+ \lambda b$ contradicting the minimality of $l$. Therefore the
vectors $\{a(g_1), a(g_2), \cdots, a(g_l) \}$ should be linearly
independent.  This completes the proof of the first statement in the
lemma. The second statement follows by a straightforward duality
argument.

\subsection{Proof of lemma \ref{lemma2}}

First note that the second part of this lemma follows by dualising the first part; here we also
use the fact that $(\Omega^l\,k)^*\cong \Omega^{-l}\,k$ which will be proved in the next lemma.
So it is enough to prove the first part. Although this lemma is secretly hidden in Benson's
treatment \cite[Theorem 4.3.2]{ben-1}, it is hard very to extract it. So we give a clean proof of this lemma using
almost split sequences, a.k.a Auslander-Reiten sequences. Recall that a short exact sequence
\[ 0 \rar A \stk{f} B \rar C \rar 0\]
of finitely generated modules over a group $G$ is an almost split sequence if it is a non-split
sequence with the property that every  map out of $A$ which is not split
injective factors through $f$. It has been shown in \cite{aus-rei-sma} that
given an finitely generated indecomposable non-projective $kG$-module $C$, there exists a unique
(up to isomorphism of short exact sequences) almost split sequence terminating
in $C$. In particular,  if $G = V_4$ and $C = \Omega^l\,k$, these sequences
are of the form; see \cite[Appendix, p 180]{ben-trends}.
\[0 \rar \Omega^{l+2}\,k \rar \Omega^{l+1}\, k \oplus \Omega^{l+1}\, k \rar \Omega^l\,k \rar 0   \hspace{9 mm}   l \ne -1\]
\[ 0 \rar \Omega^1\,k \rar kV_4 \oplus k \oplus k \rar \Omega^{-1}\,k \rar 0  \hspace{20 mm}\]
To start the proof, let $l$ be the smallest positive integer such that $\Omega^l\,k$ embeds in
a projective-free $V_4$-module $M$. If this embedding does not split, then by the property of an
almost split sequence, it should factor through $\Omega^{l-1}\,k \oplus \Omega^{l-1}\,k$ as shown in the
diagram below.
\[
\xymatrix{ 0 \ar[r] & \Omega^{l}\,k \ar[r] \ar@{^{(}->}[d] & \Omega^{l-1}\,k
\oplus \Omega^{l-1}\,k
\ar[r] \ar@{..>}[dl]^{f \oplus g} \ar[r]& \Omega^{l-2}\,k \ar[r] & 0 \\
& M & & & & }
\]
Now if either $f$ or $g$ is injective, that would contradict the minimality of $l$,
so they cannot be injective. So both $f$ and $g$ should factor through $\Omega^{l-2}\,k
\oplus \Omega^{l-2}\,k$ as shown in the diagrams below.
\[
\xymatrix{ 0 \ar[r] & \Omega^{l-1}\,k \ar[r] \ar[d]_f & \Omega^{l-2}\,k \oplus
\Omega^{l-2}\,k
\ar[r] \ar@{..>}[dl]^{(f_1 \oplus f_2)} \ar[r] & \Omega^{l-3}\,k \ar[r] & 0 \\
& M & & & & }
\]
\[
\xymatrix{ 0 \ar[r] & \Omega^{l-1}\,k \ar[r] \ar[d]_g & \Omega^{l-2}\,k \oplus
\Omega^{l-2}\,k
\ar[r] \ar@{..>}[dl]^{(g_1 \oplus g_2)} \ar[r] & \Omega^{l-3}\,k \ar[r] & 0 \\
& M & & & & }
\]

 Proceeding in this way we can assemble all the lifts obtained
using the almost split sequences into one diagram as shown below.
\[
\xymatrix{
\Omega^l\, k \ar@{^{(}->}[rrrrrr] \ar@{^{(}->}[d] & &&&&& M \\
\Omega^{l-1}\,k \oplus \Omega^{l-1}\,k \ar@{..>}[urrrrrr]  \ar@{^{(}->}[d]& &&&&& \\
(\Omega^{l-2}\,k \oplus \Omega^{l-2}\,k)\oplus(\Omega^{l-2}\,k \oplus
\Omega^{l-2}\,k) \ar@{..>}[uurrrrrr] \ar@{^{(}->}[d] \\
\vdots \ar@{^{(}->}[d] \\
(\Omega^1\, k \oplus \Omega^l\,k)\oplus \cdots \oplus(\Omega^1\, k \oplus
\Omega^l\,k) \ar@{..>}[uuuurrrrrr]  \ar@{^{(}->}[dd]\\
\\
 (kV_4 \oplus k \oplus k) \oplus \cdots \oplus (kV_4 \oplus k \oplus k)
\ar@{..>}[uuuuuurrrrrr]}
\]
So it suffices to show  that for a  projective-free $M$ there cannot exist a
factorisation of the form
\[
\xymatrix{
 \Omega^l\, k \ar@{^{(}->}[r] \ar@{^{(}->}[d] & M \\
 (kV_4)^s \oplus k^t \ar@{.>}[ur]_{\phi} & }
\]
where $l$ is a positive integer. It is not hard to see that the invariance $(\Omega^l\,k)^G$ of $\Omega^l\,k$
maps into $((kV_4)^s)^G$. We will arrive at a contradiction by showing
$((kV_4)^s)^G$ maps to zero under the map $\phi$. Since $((kV_4)^s)^G \cong
((kV_4)^G)^s$ it is enough to show that $\phi$ maps each $(kV_4)^G$ to zero.
$(kV_4)^G$ is a one-dimensional subspace, generated by say $v$. It $v$ maps to
a non-zero element, then it is easy to see that the restriction of $\phi$ on
the corresponding copy of $kV_4$ is injective, but $M$ is projective-free, so
this is impossible. In other words $\phi(v) = 0$ and that completes the proof
of the lemma.

\subsection{Proof of lemma \ref{lemma3}}
Recall that $\Omega^1(k)$ is defined to be the kernel of the
augmentation map $kV_4 \rar k$. Dualising the short exact sequence
\[ 0 \rar \Omega^1(k) \rar kV_4 \rar k \rar 0,\]
we get
\[ 0 \leftarrow \Omega^1(k)^* \leftarrow kV_4 \leftarrow k \leftarrow 0\]
because $kV_4$ and $k$ are self-dual. This shows that $\Omega^{-1}(k
) \cong \Omega^1(k)^*$. Now a straightforward induction gives
$\Omega^{-l}(k ) \cong \Omega^l(k)^*$ for all $l \ge 1$.

So it is enough to prove the part (1) of the lemma because it is not
hard to see that the representations in part (2) are precisely the
duals of those in part (1). We leave this as an easy exercise to the
reader.

As for (1) we will prove the cases $n =1$ and $n=2$. The general
case will then be abundantly clear. For $n=1$, we have to identify
the kernel of the augmentation map $kV_4 \rar k$ which is defined by
mapping the generator $e_0$ of $kV_4$ to the basis  element $g_0$ of
$k$, so the kernel  $\Omega^1(k)$ is a three dimensional
representation as shown in the diagram below
\[
\xymatrix{ \\ 0 \\  } \xymatrix{\\ \longrightarrow \\}
\xymatrix@=2em{                   \\
\overset{a_0}{\bullet} \ar[dr] & & \overset{b_0}{\bullet} \ar[dl] \\
 & \underset{c_0}{\bullet} &
} \xymatrix{\\ \longrightarrow \\} \xymatrix@=2em{
     & \overset{e_0}{\bullet} \ar[dl] \ar[dr] & \\
     \overset{a_0}{\bullet} \ar[dr]&  & \overset{b_0}{\bullet} \ar[dl]\\
     & \underset{c_0}{\bullet} &
    }
\xymatrix{\\ \longrightarrow \\} \xymatrix{\overset{g_0}{\bullet} \\ \\ } \xymatrix{\\
\longrightarrow
\\}
\xymatrix{\\ 0
 \\}
\]
Now consider the case $n=2$. Note the $\Omega^1(k)$ is generated by
two elements $g_0$ and $g_1$. So a minimal projective cover will be
$kV_4 \oplus kV_4$ generated by $e_0$ and $e_1$. The projective
covering maps $e_i$ to $g_i$, $i = 0, 1$. The kernel $\Omega^2(k)$
of this projective covering  will be $5$-dimensional and can be
easily seen in the diagram below.
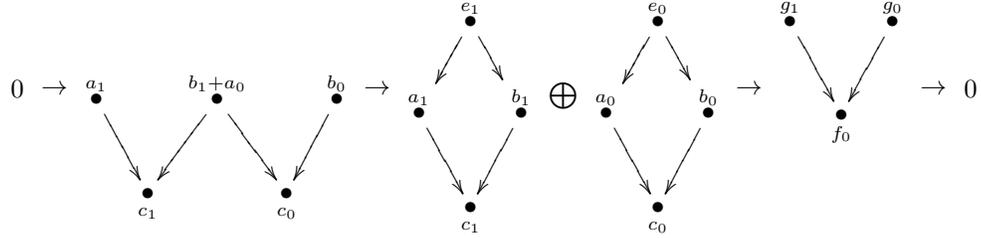
\[
\xymatrix@=0.5em{ \\ \\  0 \\ \\  } \xymatrix@=0.5em{ \\ \\  \rar \\
\\ }\xymatrix@=0.5em{ & & & & \\ \\
 \overset{a_1}{\bullet} \ar[ddr] &         & \overset{b_1 + a_0}{\bullet} \ar[ddl] \ar[ddr] &   & \overset{b_0}{\bullet} \ar[ddl]
 \\ \\
                 &\underset{c_1}{\bullet} &                         & \underset{c_0}{\bullet} &
} \xymatrix@=0.5em{\\ \\ \rar
\\ \\}
\xymatrix@=0.5em{
     & \overset{e_1}{\bullet} \ar[ddl] \ar[ddr] & \\  \\
     \overset{a_1}{\bullet} \ar[ddr]&  & \overset{b_1}{\bullet}
     \ar[ddl]\\ \\
     & \underset{c_1}{\bullet} &
    }
\xymatrix@=0.5em{\\  \\ \bigoplus \\ \\} \xymatrix@=0.5em{
     & \overset{e_0}{\bullet} \ar[ddl] \ar[ddr] & \\ \\
     \overset{a_0}{\bullet} \ar[ddr]&  & \overset{b_0}{\bullet}
     \ar[ddl]\\ \\
     & \underset{c_0}{\bullet} & }
\xymatrix@=0.5em{\\ \\ \rar \\ \\} \xymatrix@=0.5em{
\overset{g_1}{\bullet} \ar[ddr] & & \overset{g_0}{\bullet} \ar[ddl]
\\ \\
 & \underset{f_0}{\bullet} & \\ \\
} \xymatrix@=0.5em{\\ \\ \rar
\\ \\}
\xymatrix@=0.5em{\\  \\ 0
 \\ \\}
\]
Now it is clear that in general $\Omega^l(k)$ for $l \ge 1$
 will be a $2l+1$ dimensional representation and has the shape of
 the zig-zag diagram
 as shown in the statement of the lemma.

\subsection{Proof of proposition \ref{prop:heller}.}

We begin by showing the modules in item 2(b) are fixed by the Heller shift operator.
Recall that these have the form
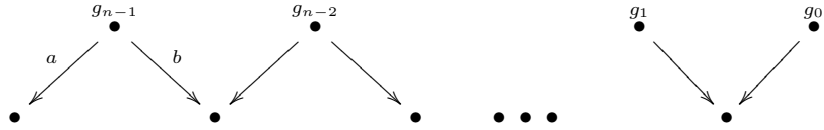
\[
     \xymatrix@=2em{
& & \overset{g_{n-1}}{\bullet} \ar[dr]^b \ar[dl]_a & &
\overset{g_{n-2}}{\bullet} \ar[dl] \ar[dr] & & &
     \overset{g_1}{\bullet} \ar[dr] & & \overset{g_0}{\bullet} \ar[dl]   \\
 &     \bullet &  & \bullet & & \bullet & \bullet  \hspace{1 mm}
      \bullet \hspace{1 mm} \bullet & &  \bullet &
}
\]
It is clear that the  $\{g_0, g_1, g_2, \cdots g_{n-1} \}$ is a
minimal generating set for the above module, $M$ say. So a minimal
projective cover of this module will be a free $V_4$-module of rank
$n$ generated by basis elements $\{e_0, e_1, e_2, \cdots ,e_{n-1}
\}$, and the covering map sends $e_i$ to $g_i$, for all $i$.
Counting dimensions, it is then clear that the dimension of the
kernel ($\Omega(M)$) of this projective cover is of dimension $2n$.
We only have to show that the $V_4$-module structure on the kernel
is isomorphic to the one on $M$. This will be clear from the
following diagrams. We consider the cases $n=2$ and $3$, the general
case will then be clear.
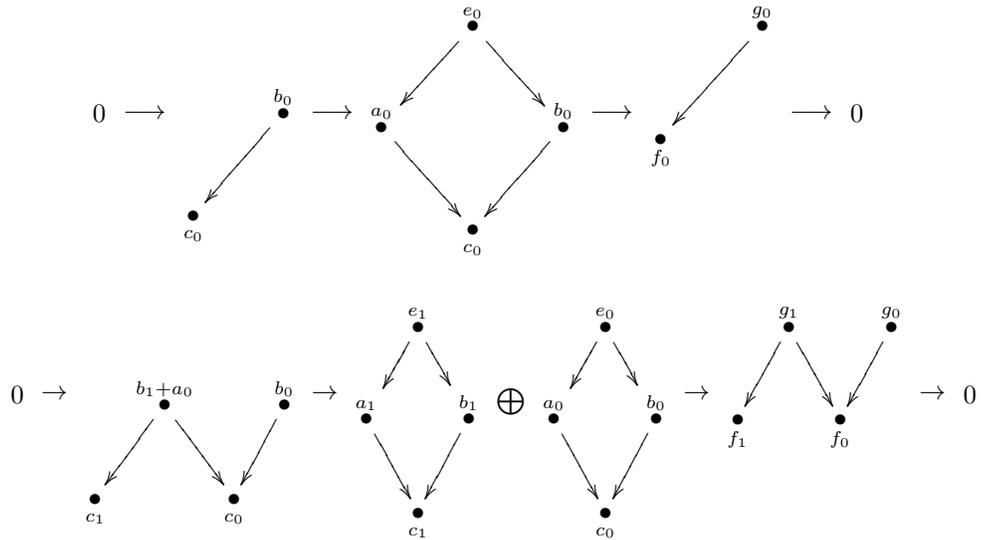
\[
\xymatrix{ \\ 0 \\  } \xymatrix{\\ \longrightarrow \\}
\xymatrix@=2em{                   \\
                 & \overset{b_0}{\bullet} \ar[dl] \\
  \underset{c_0}{\bullet} &
} \xymatrix{\\ \longrightarrow \\} \xymatrix@=2em{
     & \overset{e_0}{\bullet} \ar[dl] \ar[dr] & \\
     \overset{a_0}{\bullet} \ar[dr]&  & \overset{b_0}{\bullet} \ar[dl]\\
     & \underset{c_0}{\bullet} &
    }
\xymatrix{\\ \longrightarrow \\} \xymatrix{& \overset{g_0}{\bullet} \ar[dl] \\ \underset{f_0}\bullet & \\ } \xymatrix{\\
\longrightarrow
\\}
\xymatrix{\\ 0
 \\}
\]
\[
\xymatrix@=0.5em{ \\ \\  0 \\ \\  } \xymatrix@=0.5em{ \\ \\  \rar \\
\\ }\xymatrix@=0.5em{  & & & \\ \\
          & \overset{b_1 + a_0}{\bullet} \ar[ddl] \ar[ddr] &   & \overset{b_0}{\bullet} \ar[ddl]
 \\ \\
                 \underset{c_1}{\bullet} &                         & \underset{c_0}{\bullet} &
} \xymatrix@=0.5em{\\ \\ \rar
\\ \\}
\xymatrix@=0.5em{
     & \overset{e_1}{\bullet} \ar[ddl] \ar[ddr] & \\  \\
     \overset{a_1}{\bullet} \ar[ddr]&  & \overset{b_1}{\bullet}
     \ar[ddl]\\ \\
     & \underset{c_1}{\bullet} &
    }
\xymatrix@=0.5em{\\  \\ \bigoplus \\ \\} \xymatrix@=0.5em{
     & \overset{e_0}{\bullet} \ar[ddl] \ar[ddr] & \\ \\
     \overset{a_0}{\bullet} \ar[ddr]&  & \overset{b_0}{\bullet}
     \ar[ddl]\\ \\
     & \underset{c_0}{\bullet} & }
\xymatrix@=0.5em{\\ \\ \rar \\ \\} \xymatrix@=0.5em{ &
\overset{g_1}{\bullet} \ar[ddl] \ar[ddr] & & \overset{g_0}{\bullet}
\ar[ddl]
\\ \\
 \underset{f_1}{\bullet} & & \underset{f_0}{\bullet} & \\ \\
} \xymatrix@=0.5em{\\ \\ \rar
\\ \\}
\xymatrix@=0.5em{\\  \\ 0
 \\ \\}
\]
We now show that the modules in item 2(a) are fixed under the
Heller. Recall that in each even dimension $2n$, these modules
correspond to indecomposable rational canonical forms given by
powers of an irreducible polynomials $f(x)^l = \sum_{i=0}^n \theta_i
x^i$, schematically:
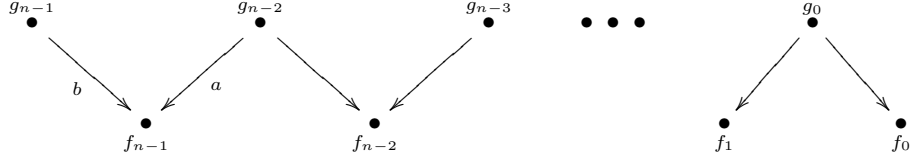
\[
\xymatrix@=1.9em{
      \overset{g_{n-1}}{\bullet} \ar[dr]_b & & \overset{g_{n-2}}{\bullet}\ar[dl]^a \ar[dr] &  & \overset{g_{n-3}}{\bullet} \ar[dl]  & \bullet  \hspace{1 mm}
      \bullet \hspace{1 mm} \bullet
      & & \overset{g_0}{\bullet} \ar[dl] \ar[dr] \\
      & \underset{f_{n-1}}{\bullet} & & \underset{f_{n-2}}{\bullet} &  & &  \underset{f_1}{\bullet} & &
      \underset{f_0}{\bullet}
    }
\]
where $a(g_{n-1}) = \sum_{i = 0}^{n-1} \theta_i f_i$. It is again
clear that  $\{g_0, g_1, g_2, \cdots g_{n-1} \}$ is a minimal
generating set, and hence a projective cover can be taken to be a
free $V_4$-module of rank $n$ with basis elements $\{e_0, e_1, e_2,
\cdots e_{n-1} \}$, and the mapping sends the elements $e_i$ to the
generators $g_i$. We will again convince the reader that these
modules are fixed under the Heller by examining the cases $n=1$ and
$n=2$. We begin with the case $n=1$. Here the rational canonical
form is determined by constant $\theta_0$, and $a(g_0) = \theta_0
f_0$. The following diagram shows that the Heller fixes these two
dimensional modules.

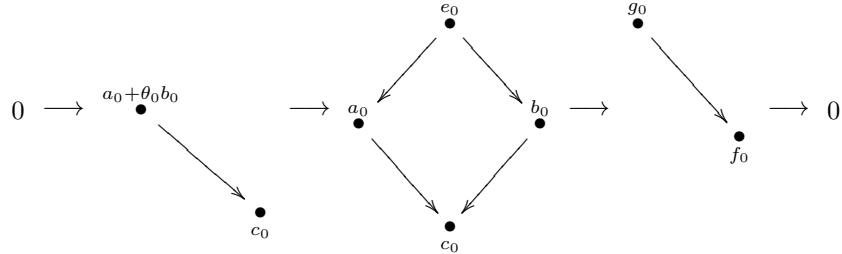
\[
\xymatrix{ \\ 0 \\  } \xymatrix{\\ \longrightarrow \\}
\xymatrix@=2em{                   \\
                  \overset{a_0 + \theta_0 b_0}{\bullet} \ar[dr] & \\
  & \underset{c_0}{\bullet}
} \xymatrix{\\ \longrightarrow \\} \xymatrix@=2em{
     & \overset{e_0}{\bullet} \ar[dl] \ar[dr] & \\
     \overset{a_0}{\bullet} \ar[dr]&  & \overset{b_0}{\bullet} \ar[dl]\\
     & \underset{c_0}{\bullet} &
    }
\xymatrix{\\ \longrightarrow \\} \xymatrix{ \overset{g_0}{\bullet} \ar[dr] & \\ & \underset{f_0}\bullet  \\ } \xymatrix{\\
\longrightarrow
\\}
\xymatrix{\\ 0
 \\}
\]
Now consider the four dimensional modules:  $n=2$ and the rational
conical form corresponds to a polynomial $x^2 + \theta_1 x +
\theta_0$. In the diagram below $a(g_1) = \theta_0 f_0  + \theta_1
f_1.$
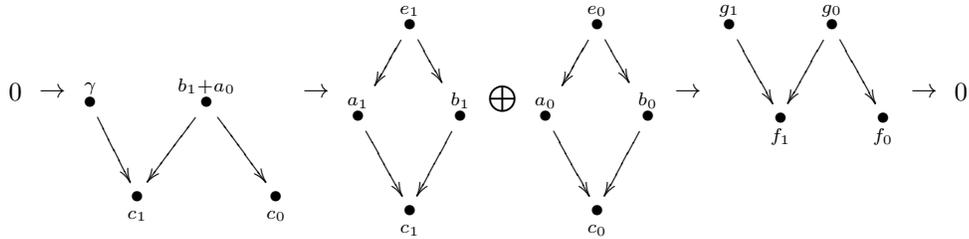
\[
\xymatrix@=0.5em{ \\ \\  0 \\ \\  } \xymatrix@=0.5em{ \\ \\  \rar \\
\\ }
\xymatrix@=0.5em{  & & & \\ \\
           \overset{\gamma}{\bullet} \ar[ddr] &   & \overset{b_1 + a_0}{\bullet}
\ar[ddl] \ar[ddr] &
 \\ \\
 &                 \underset{c_1}{\bullet} &   & \underset{c_0}{\bullet}
} \xymatrix@=0.5em{\\ \\ \rar
\\ \\}
\xymatrix@=0.5em{
     & \overset{e_1}{\bullet} \ar[ddl] \ar[ddr] & \\  \\
     \overset{a_1}{\bullet} \ar[ddr]&  & \overset{b_1}{\bullet}
     \ar[ddl]\\ \\
     & \underset{c_1}{\bullet} &
    }
\xymatrix@=0.5em{\\  \\ \bigoplus \\ \\} \xymatrix@=0.5em{
     & \overset{e_0}{\bullet} \ar[ddl] \ar[ddr] & \\ \\
     \overset{a_0}{\bullet} \ar[ddr]&  & \overset{b_0}{\bullet}
     \ar[ddl]\\ \\
     & \underset{c_0}{\bullet} & }
\xymatrix@=0.5em{\\ \\ \rar \\ \\} \xymatrix@=0.5em{
\overset{g_1}{\bullet} \ar[ddr] & & \overset{g_0}{\bullet} \ar[ddr]
\ar[ddl] &
\\ \\
 & \underset{f_1}{\bullet} & & \underset{f_0}{\bullet}  \\ \\
} \xymatrix@=0.5em{\\ \\ \rar
\\ \\}
\xymatrix@=0.5em{\\  \\ 0
 \\ \\}
\]
where $\gamma = a_1 + \theta_0 b_0 + \theta_1 b_1$. Note that
$a(\gamma) = \theta_0 c_0 + \theta_1 c_1$, as desired.

\subsection{Proof of proposition \ref{prop:dual}.}
Note that it suffices to show that the indecomposable
representations in item 2(a) are self-dual; for that forces the
representations in item 2(b) to be self-dual, and it is well known
that $kV_4$ is self-dual.

A $2n$ dimensional representation $M$ of item 2(a) can be chosen to
be of the form (after a suitable choice of bases)
\[ M = (V, V; I, J)\]
where $V$ is an $n$-dimensional vector space, $I$ denotes the
identity transformation, and $J$ an indecomposable rational
canonical form. It is then clear that the dual of $M$ is given by
\[ M^* = (V^*, V^*; I , J^T )\]
It is a interesting exercise \footnote{Hint: Use Jordan decomposition} to show that a square matrix is similar to its transpose, so
there exists an invertible matrix $D$ such that $J^T = D J D^{-1}$.
The following commutative diagram then tells us that $M$ is
isomorphic to $M^*$.
\[
\xymatrix{
  V \ar[r]^J \ar[d]_D^{\cong}  & V  \ar[d]^D_{\cong} \\
  V^* \ar[r]_{J^T} &   V^*  }
\]

\section{The quest continues}

In our paper we concentrated on Klein group but what about $C_{3} \oplus C_{3}$? What are all the representation of this group?  Interestingly enough this is an extremely difficult question. Yet, some progress has be made very recently which involves more sophisticated machinery of representation theory. For the curious reader we refer to a recent paper \cite{CFP}.


\bibliographystyle{plain}
\bibliography{lit}

\end{document}